\documentclass[11pt]{amsart}
\usepackage{amsfonts,amssymb,amsmath,amsthm}
\usepackage{url}
\usepackage{enumerate}
\usepackage[margin=3.9 cm]{geometry}

\def\s{\mathbb S}
\def\C{\mathbb C}
\def\B{\mathbb B}
\def\C{\mathbb C}

\def\U{\mathbb U}
\def\V{\mathbb V}
\def\W{\mathbb W}
\def\I{{\rm I}}

\def\C{\mathbb C}
\def\h{{\rm H}_{\C}}
\def\R{\mathbb R}
\def\U{\mathbb U}
\def\V{\mathbb V}

\def\P{\mathbb P}

\def\PU{{\rm PU}}
\def \SU{{\rm SU}}
\def \U{{\rm U}}
\def \u{\mathbb U}
\newtheorem{theorem}{Theorem}[section]
\newtheorem{lemma}[theorem]{Lemma}

\theoremstyle{definition}
\newtheorem{definition}[theorem]{Definition}

\theoremstyle{remark}

\numberwithin{equation}{section} \theoremstyle{plain}

\newtheorem{cor}[theorem]{Corollary}

\numberwithin{equation}{section}

\newcommand{\secref}[1]{section~\ref{#1}}
\newcommand{\thmref}[1]{Theorem~\ref{#1}}
\newcommand{\lemref}[1]{Lemma~\ref{#1}}

\newcommand{\corref}[1]{Corollary~\ref{#1}}
\newcommand{\eqnref}[1]{~{\textrm(\ref{#1})}}

\begin{document}
\title[Decomposition of complex hyperbolic isometries by involutions]{Decomposition of complex hyperbolic isometries by involutions}
\author[Krishnendu Gongopadhyay \and Cigole Thomas]{Krishnendu Gongopadhyay \and Cigole Thomas}
\address{Department of Mathematical Sciences, Indian Institute of Science Education and Research (IISER) Mohali, Knowledge City, S.A.S. Nagar, Sector 81, P.O. Manauli 140306, India}
\email{krishnendug@gmail.com, krishnendu@iisermohali.ac.in}

\address{Department of Mathematical Sciences, Indian Institute of Science Education and Research (IISER) Mohali, Knowledge City, S.A.S. Nagar, Sector 81, P.O. Manauli 140306, India}

\address{{\it Current Address:} Department of Mathematical Sciences, George Mason University,
4400 University Drive,
Fairfax, Virginia  22030, USA}
\email{cigolethomas@gmail.com}
\thanks{Gongopadhyay acknowledges NBHM  grant, ref no. NBHM/R.P.7/2013/Fresh/992. }
\thanks{Thomas acknowledges the support of INSPIRE-SHE scholarship.} 
\date{March 7, 2016}

\begin{abstract}
A $k$-reflection of the $n$-dimensional complex hyperbolic space ${\rm H}_{\C}^n$ is an element in ${\rm U}(n,1)$ with negative type eigenvalue $\lambda$, $|\lambda|=1$, of multiplicity $k+1$ and positive type eigenvalue $1$ of multiplicity $n-k$. We prove that a holomorphic isometry of ${\rm H}_{\C}^n$ is a product of at most four involutions and a complex $k$-reflection, $k \leq 2$. Along the way, we  prove that every element in ${\rm SU}(n)$ is a product of four or five involutions according as $n \not \equiv 2 \mod 4$ or $n \equiv 2 \mod 4$. We also give a short proof of the well-known result that every holomorphic isometry of ${\rm H}_{\C}^n$ is a product of two  anti-holomorphic involutions.
\end{abstract}
\maketitle

\section{Introduction}
An element $g$ in a group $G$ is called an involution if $g^2=1$. An element $g$ in $G$ is called \hbox{\emph{reversible}} if $g$ is conjugate to $g^{-1}$. An element that is a product of two involutions is called strongly reversible. The reversible,  strongly reversible elements and their relationship have been investigated in several contexts in the literature, for example see \cite{e1, e2, kn1, kn2,  fr, st1, wo}. In \cite{gp}, Gongopadhyay and Parker classified the reversible and strongly reversible isometries of the $n$-dimensional complex hyperbolic space. Classification of orientation-preserving reversible isometries of the real hyperbolic space was obtained in \cite{g}. A related question is to obtain the minimum number $k$ of involutions that is required to generate an  element $g$ in a group $G$; the number $k$ is called the ``involution length" of $g$. The maximum of all involution lengths over elements of $G$ is the involution length of the group $G$. This question was investigated and settled for orthogonal groups over arbitrary fields by Ellers \cite{e1}, Nielsen \cite{n1} and, Kn\"uppel and Nielsen \cite{kn1}, also see \cite{kt} where the authors have also investigated commutator width of orthogonal transformations. Recently, Basmajian and Maskit \cite{bm} have settled this question for isometries of the space-forms: the Euclidean $n$-space, the $n$-sphere and the real hyperbolic $n$-space, also see \cite{bp} for related work. It is natural to ask for the same question in unitary groups. However, in unitary groups situation is more complicated as there are complex reflections that are not involutions. B\"unger and Kn\"uppel \cite{bn1} have investigated decompositions of unitary transformations. They proved that every unitary transformation over an algebraically closed field is a product of three `quasi-involutions'. 
Djokovi\'c    and Malzan \cite{dm} investigated the problem for unitary groups $\U(p,q)$ over complex numbers and proved that an element $g$ of $\U(p,q)$ with determinant $\pm 1$ is a product of `involutory-reflections'.  An
 involutory-reflection is an involution that fixes every point on a non-degenerate hyperplane of $\C^{p+q}$. They gave a bound of $p+q+4$ for the number of involutory-reflections that is needed to express an element $g$. 

In this paper, our interest is the isometry group $\PU(n,1)$ of the $n$-dimensional complex hyperbolic space $\h^n$.  A \emph{complex $k$-reflection} of $\h^n$ is an elliptic isometry that has an eigenvalue $1$ of multiplicity $n-k$ and an eigenvalue $\lambda$ corresponding to the fixed points on $\h^n$, of multiplicity $k+1$. A {complex reflection} need not be an involution.  It follows from the result of  B\"unger and Kn\"uppel \cite{bn1} that every element in $\PU(n,1)$ is a product of an involution and two elliptic isometries. We prove in this paper that we can take those elliptic isometries as a product of three involutions and a complex $k$-reflection. That is, we prove that every element in $\PU(n,1)$ is a product of at most four involutions and a complex $k$-reflection, $k \leq 2$,  see \thmref{th2} in \secref{chi}. Thus every isometry of $\h^n$ is a product of a {complex $k$-reflection} and two reversible elements. Along the way, we prove that the involution length of  $\SU(n)$ is four or five according as $n \not \equiv 2 \mod 4$ or $n = 2 \mod 4$, see \thmref{th1} in \secref{sun}.  Djokovi\'c and Malzan \cite{dm2} obtained a formula for the involutory-reflection length of an element with determinant $\pm 1$ in $\U(n)$ and established that the involutory-reflection length is $2n-1$. Our result shows that if instead of the family of involutory-reflections, we take the set of all involutions as the generating set, then the involution length of $\SU(n)$ is essentially independent of $n$ and can be improved further to four or five. 

We have learned that Julien Paupert and Pierre Will \cite{pw2} are investigating involution length in $\PU(n,1)$ and it seems to them that the involution length of $\PU(2,1)$ is $4$.  As a consequence of the work in this paper,  the problem of finding involution length in $\PU(n,1)$ is now closely related to the problem of finding involution length of $k$-reflections, $k \leq 2$, also see \lemref{invo} where it has been observed that the involution length is also closely related to the ``Hermitian length" of an element. 

Finally, in \secref{choi}, we give a short proof of a well-known result by Choi \cite{choi} that states that every holomorphic isometry of $\h^n$ is a product of two anti-holomorphic involutions. Choi's original proof is not available in literature and the result for $\PU(2,1)$ was proved by Falbel and Zocca \cite{fz} using a different argument. This result is a starting point of the recent investigation of Paupert and Will on ``linked pairs by real reflections" \cite{pw1}.

\section{Preliminaries} 
All the assertions made in this section are borrowed essentially from \cite{chen}. 

\subsection{The Complex Hyperbolic Space} Let $\V = \C^{n+1}$ be a vector space of dimension $(n+1)$ over $\C$ equipped with the complex Hermitian form of \emph{signature} $(n,1)$,
$$
\langle  z,w \rangle  = {\overline w}^t J z=-z_0\overline{w}_0 + z_1 \overline{w}_1 + \cdots  + 
z_n \overline{w}_n,
$$
where $z$ and $w$ are the column vectors in $\V$ with entries
$z_0,\ldots,z_{n}$ and $w_0,\ldots,w_{n}$ respectively and $J$ is the diagonal matrix $J=diag(-1, 1, \ldots, 1)$ representing the Hermitian form. We denote $\V$ by $\C^{n,1}$.  Define
$$
\V_0=\{z \in \V \ |\;\langle  z,z \rangle  =0\}, \ \V_+=\{z \in \V \ |\;\langle  z,z \rangle >  0\},\ \V_-=\{z \in \V \ |\;\langle z,z\rangle< 0\}.
$$
A vector $v$ is called \emph{time-like}, \emph{space-like} or \emph{light-like} according as $v$ is an element in $\V_-$, $\V_+$ or $\V_0$. 
Let $\P(\V)$ be the projective space obtained from $\V$, i.e.,
$\P(\V)=\V-\{0\}/\sim$, where $u \sim v$ if there exists $\lambda$
in $\C^{\ast}$ such that $u=v \lambda$, and $\P(\V)$ is equipped
with the quotient topology. Let $\pi: \V-\{0\}\to \P(\V)$ denote the projection map. We define $\h^n=\pi(\V_-)$. The boundary $\partial \h^n$ in $\P(\V)$ is
$\pi(\V_0-\{0\})$.  The unitary group ${\rm U}(n,1)$ of the
Hermitian space $\V$ acts by the holomorphic isometries of $\h^n$. 

A matrix $A$ in ${\rm GL}(n+1,{\mathbb C})$ is \emph{unitary} with respect to the Hermitian form $\langle
z,w\rangle$ if
$\langle Az,Aw\rangle=\langle z,w\rangle$ for all $z,\,w\in{\mathbb V}$.
Let ${\rm U}(n,1)$ denote the group of all matrices that are unitary
with respect to our Hermitian form of signature $(n,1)$. By letting
$z$ and $w$ vary through a basis of ${\mathbb V}$ we can characterize
${\rm U}(n,1)$  by
$${\rm U}(n,1)= \{A \in {\rm GL}(n+1, \C) \ : \  {\bar A}^t J A = J \}.$$
The group of isometries of $\h^n$  is ${\rm PU}(n,1)={\rm U}(n,1)/Z({\rm U}(n,1)),$ where
the center $Z({\rm U}(n,1))$ can be identified with the circle group $\s^1=\{\lambda I \ |  \ |\lambda|=1 \}$.  Let { ${\rm O}(n,1)={\rm GL}(n+1, {\mathbb \R})\cap{\rm U}(n,1)$}. Then ${\rm PO}(n,1)$ is the isometry group of the real hyperbolic $n$-space that is embedded inside $\h^n$.

Thus an isometry $T$ of  $\h^n$ lifts to a unitary transformation $\tilde{T}$ in ${\rm U}(n,1)$ and  
the fixed points of $T$ correspond to eigenvectors of
$\tilde{T}$. For our purpose, it is convenient to deal with ${\rm U}(n,1)$ rather than ${\rm PU}(n,1)$. We shall regard ${\rm U}(n,1)$ as acting on $\h^n$ as well as on $\V$. 

A subspace $\W$ of $\V$ is called  \emph{space-like}, \emph{light-like}, or \emph{ indefinite} if the Hermitian form restricted to $\W$ is positive-definite, degenerate, or non-degenerate but indefinite respectively. If $\W$ is an indefinite subspace of $\V$, then the orthogonal complement $\W^{\perp}$  is space-like. 

\medskip The \emph{ball model} of $\h^n$ is obtained by taking the representatives of the homogeneous  coordinate  $Z=[(1, z_1, \ldots, z_n)]$ in $\pi(\V)$. The vector $(1, z_1, \ldots, z_n)$ is the \emph{standard lift} of ${z} \in \h^n$ to $\V_-$. Further we see that ${ z} \in \h^n$ provided
$$\langle Z, Z \rangle= -1+|z_1|^2 + \cdots +|z_n|^2<0,$$
i.e. $|z_1|^2 + \cdots +|z_n|^2<1$. Thus $\pi(\V_-)$ can be identified with the unit ball 
$$\B^n=\{(z_1, \cdots, z_n) \in \C^n \ : \ |z_1|^2 + \cdots + |z_n|^2<1 \}.$$
This identifies the boundary $\partial \h^n$ with the \emph{ unit sphere}
$$\s^{2n-1}=\{(z_1, \cdots, z_n) \in \C^n \ : \ |z_1|^2 + \cdots + |z_n|^2=1 \}.$$

 In the ball model of the hyperbolic space, by Brouwer's fixed point theorem it follows 
that every isometry $T$ has a fixed point on the closure $\overline{\h^n}$. An isometry $T$ is called \emph{elliptic} if it has a fixed point in $\h^n$;
it is called \emph{parabolic} if it fixes a single point and this point lies 
in $\partial\h^n$; it is called \emph{hyperbolic} (or \emph{loxodromic}) if 
it fixes exactly two points and they both lie on $\partial \h^n$. Any
non-central element $T$ of ${\rm U}(n,1)$ must be one of the above three 
types; see \cite{chen}. 

\subsection{Conjugacy classification of isometries}
    It follows from the conjugacy classification in ${\rm U}(n,1)$, see \cite[Theorem 3.4.1]{chen},  that the elliptic and hyperbolic elements are semisimple, i.e. their minimal polynomial is a product of linear factors. The parabolic elements are not semisimple. A parabolic transformation $T$ has the unique Jordan decomposition $T=AN$, where  $A$ is elliptic, $N$ is unipotent and $AN=NA$. 

\begin{definition} 
 An eigenvalue $\lambda$ of $T \in {\rm U}(n,1)$ is said to be of \emph{negative type}, of \emph{positive type}  if every eigenvector in $\V_\lambda$ is in $\V_-$ or $\V_+$ respectively. The eigenvalue $\lambda$ is called \emph{null} if the $\lambda$-eigenspace $\V_{\lambda}$ is light-like. The eigenvalue $\lambda$ is said to be of \emph{indefinite type} if $\V_\lambda$ contains vectors in $\V_-$ and
vectors in $\V_+$. Moreover, for $\lambda$ of indefinite type,  the restriction of the Hermitian form to $\V_{\lambda}$ has signature $(r,1)$, $1 \leq r \leq n$,  where ${\rm \dim } \  \V_{\lambda}=r+1$. 
\end{definition} 

Let $T$ be elliptic. From the conjugacy classification it follows that all eigenvalues of $T$ except for one are of positive type and the remaining eigenvalue is either of negative type or of indefinite type. Moreover, all eigenvalues of $T$ have modulus $1$. 

Suppose $T$ is hyperbolic. Then it has a pair of null eigenvalues $re^{ i \theta}$, $r^{-1} e^{ i \theta}$, $r>1$,  and the eigenspace of each such eigenvalue has dimension one.  The other eigenvalues are of positive type and they all have modulus one.  

Suppose $T$ is parabolic. If $T$ is unipotent, i.e.
all the eigenvalues are $1$, then it has minimal polynomial $(x-1)^2$, or $(x-1)^3$ and, accordingly $T$ is called \emph{vertical} or \emph{non-vertical translation}. 

 If $T$ is a non-unipotent parabolic, then it has the Jordan
decomposition $T=AN$ as above. In this case $T$ has a null eigenvalue $\lambda$, $|\lambda|=1$,  and the minimal polynomial of $T$ contains a factor of the form $(x-\lambda)^2$ or $(x-\lambda)^3$. This implies that $\C^{n,1}$ has a $T$-invariant orthogonal decomposition 
\begin{equation}\label{pe} \C^{n,1}=\u \oplus \W, \end{equation} where $T|_{\W}$ is semisimple, $\u$ is indefinite, $\dim \u=k$ with $k=2$ or $3$ and $T|_{\u}$ has characteristic, as well as minimal polynomial $(x-\lambda)^k$. If $k=2$, $T$ is called a \emph{ellipto-translation} and for $k=3$, $T$ is called a \emph{ellipto-parabolic}. Without loss of generality, we can assume, $T|_{\W}$ is an element in $\U(n-k+1)$ by identifying $\U(\langle, \rangle|_{\W})$ with $\U(n-k+1)$. 

We note here that  $\oplus$ will always denote the orthogonal sum of subspaces. The direct sum is denoted by $+$. 
\subsection{Complex Reflections} We slightly generalize the notion of a complex reflection.  An element $T$ of $\U(n,1)$ is called a \emph{complex $k$-reflection} if it has a negative eigenvalue $\lambda$ of multiplicity $k+1$ and $n-k$ eigenvalues $1$.  In the ball model of $\h^n$, a complex $0$-reflection is simply a transformation of the form $Z \mapsto \lambda Z$, $|\lambda|=1$. A $0$-reflection is called a \emph{complex rotation} of $\h^n$. A complex $k$-reflection pointwise fixes a $k$-dimensional totally geodesic subspace $S$ of $\h^n$ and acts as a rotation in the co-dimension $k$ orthogonal complement of $S$. A $1$-reflection is called a \emph{complex line-reflection} and a $2$-reflection is called a \emph{complex plane-reflection}. 

Note that usually what is called a complex reflection in the literature, is our $(n-1)$-reflection. For more details on this kind of reflections and their triangle group see the survey by Parker \cite{p}. 
\subsection{Involutions in $\U(n,1)$} In this section we give a characterization of involutions in $\U(n,1)$. Though we will not use it anywhere in the sequel,  the lemma is of independent interest. This relates the problem of finding the involution length in $\PU(n,1)$ to the problem  of expressing every element in $\PU(n,1)$ as a product of Hermitian matrices. 
\begin{lemma}\label{invo}
An element $A\in \U(n,1)$ is an involution iff $A=HJ$ where $H\in \U(n,1)$ is  Hermitian and $J=\hbox{diag}(-1, 1, \cdots, 1)$ is the matrix corresponding to the Hermitian form on $\mathbb{C}^{n,1}$.
\end{lemma}
\begin{proof}
Let $A\in \U(n,1)$ be an involution. Then $A=A^{-1}$ and it follows from $AJ\bar{A}^t=J$ that $J\bar{A}^t = AJ$. As $\overline{(J\bar{A}^t)}^t=AJ$, it follows that $J\bar{A}^t$ is hermitian. Hence, $A=HJ$ where $H=J\bar{A}^t$.

Conversely, let $A=HJ$ where $H \in \U(n,1)$ is Hermitian. Then $A^2= HJHJ=HJ\bar{H}^tJ=HH^{-1}=I$.
\end{proof}
In particular it follows that: 
\begin{cor}
If $A$ is Hermitian in $\U(n,1)$, then it is strongly reversible. In particular, every Hermitian element in $\U(n,1)$ is reversible. 
\end{cor}
\begin{proof}
As $HJ=A$ is an involution, we have $H=AJ$ as a product of two involutions in $\U(n,1)$. Hence it is strongly reversible.  
\end{proof}

\section{Product of involutions in $\SU(n)$} \label{sun}
In this section we prove the following theorem.
\begin{theorem} \label{th1} Let $n>2$.  If $n \not \equiv 2 \mod 4$, an unitary transformation in $\SU(n)$ is a product of at most four involutions. If $n \equiv 2 \mod 4$, then every element in $\SU(n)$ is a product of at most five involutions. 

That is,  the involution length of $\SU(n)$ is four, resp. five, if $n \not \equiv 2 \mod 4$, resp. $n \equiv 2 \mod 4$. 
\end{theorem}

The proof of the theorem will follow from the following lemmas. 
\begin{lemma} \label{lem1}\cite{e2, gp} Let $n \not \equiv 2 \mod 4$. An element $T \in \SU(n)$ is reversible if and only if it is a product of two involutions. 
\end{lemma}
\begin{lemma} \label{lem2} 
If $n \equiv 2 \mod 4$, then a reversible element $T$ in $\SU(n)$ that has no eigenvalue $\pm 1$, can be written as a product $T=J_1 J_2$, where $J_1$ and $J_2$ are involutions in $\U(n)$, each of determinant $-1$. If $T$ has an eigenvalue $\pm 1$, it can be written as a product of two involutions in $\SU(n)$. 
\end{lemma}
\begin{proof}
Let $n=4m+2$. 
If $T \in \SU(n)$ be reversible. Then if $\lambda$ is a root, so is $\lambda^{-1}$ with the same multiplicity. Thus we can decompose $\C^n$ into two-dimensional subspaces
\begin{equation} \label{eq1} \C^n=\W_1 \oplus \W_2 \oplus \cdots \oplus \W_{2m+1},\end{equation} 
where each $\W_i$ has an orthonormal basis $w_{i1}, w_{i2}$ such that $T(w_{i1})=\lambda w_{i1}$ and $T(w_{i2})=\lambda^{-1} w_{i2}$. Define $J_1$ and $J_2$ such that their restrictions on $\W_i$ is given by
$$J_{i1}(w_{i1})=\lambda w_{i2}, ~ J_{i1}(w_{i2})=\lambda^{-1} w_{i1}; ~~~ J_{i2}(w_{i1})= w_{i2}, ~ J_{i2}(w_{i2})= w_{i1}.$$
Then for each $i=1,2, \ldots,2m+1$, $J_{i1}$ and $J_{i2}$ are involutions each with determinant $-1$. 
Let $J_1=J_{11} \oplus \cdots \oplus J_{(2m+1)1}$ and $J_2=J_{12} \oplus \cdots \oplus J_{(2m+1)2}$. 
Then $T=J_2 J_1$ and $\det J_1=-1=\det J_2$, $J_1^2=I=J_2^2$. 

If $T$ has an eigenvalue $\pm 1$, then $\C^n$ has a $T$-invariant orthogonal decomposition
$$\C^n=\u_1 \oplus \u_{-1} \oplus \W,$$
where  $\dim \u_{-1}$ is even, say $2l$, $T|_{\u_{-1}}=-1_{2l}$;  $\dim \u_1=k$, $T|_{\u_1}=1_k$ and,  $T|_{\W}$ has no eigenvalue $\pm 1$. By the above method, $T|_{\W}=j_1 j_2$ for involutions $j_1$, $j_2$ on $\W$. 
Define $J_1=-1 \oplus 1_{k-1} \oplus -1_{2l}\oplus j_1$, $J_2=-1 \oplus 1_{2l+k-1} \oplus j_2$. Then $J_1$ and $J_2$ are involutions such that each has determinant one and $T=J_2 J_1$. This completes the proof. 
\end{proof} 
\begin{lemma}\label{lem3}
Every element in $\SU(n)$, can be written as a product of two reversible elements. 
\end{lemma} 
\begin{proof}
Suppose $A$ is an element of  $\SU(n)$. Let $\lambda_1, \ldots, \lambda_n$ be the eigenvalues of $A$. Note that $|\lambda_i|=1$ for all $i$.  Then $\C^n$ has an orthogonal decomposition into eigenspaces:
$$\C^n=\V_1 \oplus \cdots \oplus \V_n, $$
where each $\V_i$ has dimension $1$ and $T|_{\V_i}(v)=\lambda_i v$ for $v \in \V_i$. Choose an orthonormal basis $v_1, v_2, \ldots, v_n$ of $\C^n$, where $v_i \in \V_i$ for each $i$.  
Consider the unitary transformations $R_1: \V \to \V$ and $R_2: \V \to \V$ defined as follows: for each $k=0, 1, 2\ldots, $
\begin{equation} \label{r1} R_1(v_{2k})=\prod_{j=0}^{2(k-1)} {\overline \lambda}_{2k-j-1} ~v_{2k},~~~~~R_1(v_{2k+1})=\prod_{j=0}^{2k} \lambda_{2k-j+1} ~v_{2k+1}, \end{equation}
\begin{equation} \label{r2} R_2(v_{2k})= \prod_{j=0}^{2k-1} {\lambda}_{2k-j} ~v_{2k},~~~~R_2(v_{2k+1})=\prod_{j=0}^{2k-1} {\overline \lambda}_{2k-j} ~v_k,\end{equation}
with the convention $\lambda_{0}=1=\lambda_{-1}, ~v_0=0$.   Note that $k \leq [\frac{n}{2}]+1$ and $\max k=\frac{n}{2}$ or $\frac{n-1}{2}$ depending on $n$ is even or odd.  For each $i$, $R_1R_2(v_i)=\lambda_i v_i=T(v_i)$, and hence $T=R_1 R_2$. Note that both $R_1$ and $R_2$ has the property that if $\lambda$ is an eigenvalue, then so is 
$\bar \lambda=\lambda^{-1}$. This shows that $R_1$ and $R_2$ are reversible, cf. \cite{gp}. Further, if $T \in \SU(n)$, then $\lambda_1 \cdots \lambda_n=1$ and hence,  both $R_1$ and $R_2$ have determinants $1$. Hence the result follows.  
\end{proof}
In matrix form, up to conjugacy, if  $T=\hbox{diag}(\lambda_1, \lambda_2, \ldots, \lambda_n)$, then
\begin{equation}\label{r3} R_1=\hbox{diag}(\lambda_1, \overline \lambda_1, ~ \lambda_1\lambda_2 \lambda_3, \overline \lambda_1 \overline \lambda_2 \overline \lambda_3, ~\ldots, ~ \lambda_1 \lambda_2 \cdots \lambda_{2k+1}, \overline \lambda_1 \overline \lambda_2\cdots \overline \lambda_{2k+1}, \ldots) \end{equation}
\begin{equation} \label{r4} R_2=\hbox{diag}(1, ~\lambda_1 \lambda_2, \overline \lambda_1 \overline \lambda_2, ~\ldots, ~\lambda_1 \lambda_2 \cdots \lambda_{2k}, \overline\lambda_1 \overline \lambda_2 \cdots \overline \lambda_{2k}, \ldots).\end{equation}
Note that $R_2$ has always an eigenvalue $1$. Hence it can be written as a product of two involutions, see \cite[Proposition 3.3]{gp}. 
\begin{lemma}\label{lem4}
Let $n \equiv 2 \mod 4$, $n >2$. Let $T \in \SU(n)$ be a reversible element that can not be written as a product of two involutions in $\SU(n)$. Then $T$ can be written as a product of three involutions in $\SU(n)$. 
 \end{lemma}
\begin{proof}
Let $n=4m+2$. We have the decomposition of $\C^n$ as in \eqnref{eq1}. Further we see that $T|_{\W_i}=J_{i1} J_{i2}$, where  $J_{i1}$ and $J_{i2}$ are involutions each with determinant $-1$. Now define involutions $I_1$, $I_2$, $I_3$ as follows.
$$I_1|_{\W_1}=J_{11}, \;\;I_1|_{\W_2}=1, \;\; I_1|_{\W_i}=J_{i1}, \;i=3, \ldots, 2m+1.$$
$$I_2|_{\W_1}=1, \;\;\;\;I_2|_{\W_2}=J_{21}, ~~~ I_2|_{\W_i}=J_{i2}, ~ i=3, \ldots, 2m+1.$$
$$I_3|_{\W_1}=J_{12}, \;\; I_3|_{\W_2}=J_{22}, ~~ ~I_3|_{\W_i}=1, ~ i=3, \ldots, 2m+1.$$
Then each $I_1$, $I_2$, $I_3$ has determinant $1$ and they are involutions, and $T=I_1 I_2 I_3$. 
\end{proof} 
\subsubsection*{Proof of \thmref{th1}}Combining the above lemmas we have \thmref{th1}. 

\section{Decomposition of complex hyperbolic isometries}\label{chi}
In this section, we prove the following theorem. 
\begin{theorem} \label{th2} Let $T$ be a holomorphic isometry of $\h^n$, that is, $T \in \PU(n,1)$. Then $T$ is a product of at most four involutions and a complex $k$-reflection, where $k \leq 2$; $k=0$ if $T$ is elliptic; $k=1$ if $T$ is ellipto-translation or hyperbolic; $k=2$ if $T$ is ellipto-parabolic and $n>2$. 
\end{theorem}
Note that ellipto-parabolics do not exists when $n=2$. An element $f$ in $\U(2,1)$ that has minimal polynomial of the form $(x-\lambda)^3$,  is an element of the form $\Lambda P$, where $\Lambda=\lambda 1_{3}$ is a central element and $P$ is a non-vertical translation. Hence $f$  acts as a non-vertical translation. 

Since an isometry that is a product of two involutions is also reversible, we have the following.
\begin{cor}
Let $T$ be a holomorphic isometry of $\h^n$, that is, $T \in \PU(n,1)$. Then $T$ is a product of at most two reversible elements and a complex $k$-reflection, where $k \leq 2$; $k=0$ if $T$ is elliptic; $k=1$ if $T$ is ellipto-translation or hyperbolic; $k=2$ if $T$ is ellipto-parabolic and $n>2$.
\end{cor} 

The theorem will follow from several lemmas that we prove below. We also note down the following theorem from \cite{gp} that will be used in the proof.

\begin{theorem} \label{rch}\cite[Theorem 4.2]{gp}
\begin{itemize}
\item[(i)] Let $T$ be an element of ${\rm U}(n,1)$. Then $T$ is strongly
reversible if and only if it is reversible.
\item[(ii)] Let $T$ be an element of ${\rm SU}(n,1)$ whose characteristic
polynomial is self-dual. Then the following conditions are equivalent
\begin{enumerate}
\item[(a)] $T$ is reversible but not strongly reversible.
\item[(b)] $T$ is hyperbolic, $n\equiv 1 \mod 4$, and $\pm1$
is not an eigenvalue of $T$.
\end{enumerate}
\end{itemize}
\end{theorem}

\medskip Suppose that $T$ is in ${\rm PU}(n,1)$. Let $\widetilde T$ be a lift of $T$ to 
${\rm U}(n,1)$ and note that $e^{i\theta}\widetilde T$ corresponds to the same 
element of ${\rm PU}(n,1)$ for all $\theta\in[0,2\pi)$. For simplicity, we will not use the `tilde'  anymore to denote the lift and will use the same symbol for an element in $\PU(n,1)$ and its preferred choice of lift. 

\begin{lemma}\label{e}
Let $T$ be an elliptic element of $\SU(n,1)$ with negative type eigenvalue $1$. Then $T$ is a product of at most four involutions. 
\end{lemma}
\begin{proof}
Since $T$ has negative type eigenvalue $1$,  $\C^{n,1}$ has a $T$-invariant decomposition $\C^{n,1}=L \oplus \W$, where $T|_L=1$, $\dim L=1$ and $\dim \W=n$, $T|_{\W} \in \SU(n)$. By \thmref{th1}, if $n \not \equiv 2 \mod 4$, then $T|_{\W}$ can be written as a product of four involutions. Assume $T|_{\W}$ has no eigenvalue $\pm 1$. If $n\equiv 2 \mod 4$, it follows from \lemref{lem2} and \lemref{lem3} that $T|_{\W}=j_1 j_2 j_3 j_4$, where $j_i$ are involutions in $\U(n)$ each of determinant $-1$. For each $i=1,2,3,4$, define $J_i=-1 \oplus j_i$. Then $J_i$ is an element of $\SU(n,1)$ and $T=J_1 J_2 J_3 J_4$.  When $T|_{\W}$ has an eigenvalue $\pm 1$, then it can be seen using \lemref{lem2} that it is a product of four involutions.  This proves the lemma. 
\end{proof}

\begin{lemma}\label{cel}
Let $T$ be an elliptic element in $\PU(n,1)$. Then $T$ is a product of a $k$-reflection, $k \geq 0$, and four involutions. 
\end{lemma}
\begin{proof}
Choose a lift of $T$ in $\U(n,1)$ such that $\C^{n,1}$ has a $T$-invariant orthogonal decomposition
$\C^{n,1}=\u \oplus \W$, where $\dim \u=k+1 \geq 1$, $T|_{\W} \in \SU(n-k)$ and $T|_{\u}(v)=\lambda v$. Thus we have $T=RK$, where $R$ is a $k$-reflection and $K \in \SU(n,1)$ with negative type eigenvalue $1$ of multiplicity $k+1$. By the above lemma it follows that $T=R J_1 J_2 J_3 J_4$. This completes the proof. 
\end{proof}
\begin{cor}  \label{ce}
Let $T$ be an elliptic element in $\PU(n,1)$. Then $T$ is a product of a complex rotation and four involutions. 
\end{cor}
\begin{proof}
Since $T$ is semisimple, we can choose a lift $T$ such that $\C^{n,1}$ has the decomposition $T=RK$, where $K \in \SU(n,1)$ be an elliptic with negative type eigenvalue $1$ and $R$ is an elliptic with one negative type eigenvalue $\lambda$,  $|\lambda|=1$, and one positive type eigenvalue $1$ of multiplicity $n$. Note that $R$ represents a complex rotation. The proof now follows as above. 
\end{proof}
\begin{lemma}
Let $T$ be a hyperbolic element in $\SU(n,1)$, $n>2$,  with real null eigenvalues. Then $T$ can be written as a product of four involutions. 
\end{lemma}
\begin{proof}
Since $T$ has null eigenvalues real numbers $r$, $r^{-1}$, hence $\C^{n,1}$ has a $T$-invariant decomposition
$$\C^{n,1}=H \oplus \W,$$
where $H=\V_r + \V_{r^{-1}}, ~ \dim \V_r=1=\dim \V_{r^{-1}}$ and $T|_{\W} \in \SU(n-1)$.  
By \lemref{lem3}, $T|_{\W}=r_1 r_2$, where $r_1$ and $r_2$ are reversible elements in $\SU(n-1)$ and are of the form given by \eqnref{r1} and \eqnref{r2}. Let 
$R_1=1|_H \oplus r_1$ and $R_2=T|_H \oplus r_2$. Then $R_1$ and $R_2$ are reversible elements in $\SU(n,1)$. Note that $R_1$ is elliptic and $R_2$ is hyperbolic with an eigenvalue $1$. By \thmref{rch}, it follows that both $R_1$ and $R_2$ can be expressed as a product of two involutions in $\SU(n,1)$. Hence $T$ can be written as a product of four involutions in $\SU(n,1)$. 
\end{proof}
\begin{cor}\label{ch}
A hyperbolic element in $\PU(n,1)$ is a product of a  complex line-reflection and four involutions. 
\end{cor}
\begin{proof}
A hyperbolic element $T$ in $\U(n,1)$ can be written as $T=DK$, where $K \in \SU(n,1)$ is a hyperbolic element with real null eigenvalues and $D$, up to conjugacy,  is a diagonal matrix of the form $\lambda 1_2 \oplus 1_{n-1}$. $D$ is clearly a   complex  line-reflection. The result now follow from the above lemma.
\end{proof}
\begin{lemma}\label{vt} 
A vertical-translation in $\PU(n,1)$, $n\geq2$ is a product of four involutions. A non-vertical translation is a product of two involutions.
\end{lemma} 
\begin{proof}
Let $T$ be a vertical translation. Without loss of generality we may assume $T \in \SU(n,1)$. Then $\C^{n,1}$ has $T$-invariant orthogonal decomposition
$$\C^{n,1}= \W \oplus \W',$$
where $\dim \W=2$ is $T$-indecomposable and $T|_{\W'}=1_{\W'}$. So, without loss of generality we can assume $T \in \SU(1,1)$. 
It follows from the theorem of  Djokovi\'c and Malzan \cite{dm} that the involutory-reflection length of $T$ is bounded by
$$l(T)=\hbox{rank}(1-T)+ \frac{1}{2}(1-{(-1)^{\hbox{rank}(1-T)} \det T)}+ \delta_2,$$ where $\delta_2$ is either $0$ or $2$. Now, clearly rank of $1-T$ is $1$ in this case. So, $l(T)=2+ \delta_2$. But $\delta_2$ can not be zero as in that case, $T$ is reversible, which is not possible by \cite[Theorem 4.1]{gp}. Hence $l(T)$ must be $4$. Consequently, the assertion follows. 

 It follows from \cite[Theorem 4.1]{gp} that a non-vertical translation is reversible. Hence,  using \thmref{rch}, the result follows. 
\end{proof}
\begin{lemma}\label{et}
Let $T$ be an ellipto-translation in $\PU(n,1)$. Then it is a product of  a complex line-reflection and four involutions. 
\end{lemma}
\begin{proof}
Choose a lift in $\U(n,1)$ such that $T=DP$, where $P$ is a ellipto-translation in $\SU(n,1)$ with null eigenvalue $1$ and, $D$ is elliptic with characteristic polynomial \hbox{$(x-\lambda)^2 (x-1)^{n-1}$}, $|\lambda|=1$. Now, $\C^{n,1}$ has a $P$-invariant decomposition $\C^{n,1}={\mathbb U} \oplus \W$, where $\dim {\mathbb U}=2$, $P|_{{\mathbb U}}$ has minimal polynomial $(x-1)^2$ and $P|_{\W} \in \SU(n-1)$. It follows as above from Djokovi\'c and Malzan's theorem that $P|_{{\mathbb U}}$ is a product $i_1 i_2 i_3 i_4$ of involutions and, by \lemref{lem2} and \lemref{lem3}, $P|_{\W}$ is a product of four involutions $r_1 r_2 r_3 r_4$. Thus $P$ is product of four involutions $R_k=i_k \oplus r_k$ in $\U(n,1)$. Clearly, $D$ is a complex line-reflection. Hence the lemma is proved.
\end{proof}
\begin{cor}
Let $T$ be an ellipto-translation in $\SU(n,1)$ with null eigenvalue $1$. Then $T$ is a product of four involutions in $\U(n,1)$. 
\end{cor}
\begin{lemma}\label{ep}
Let $T$ be an ellipto-parabolic in $\PU(n,1)$. Then it is a product of a complex plane-reflection and four involutions. 
\end{lemma}
\begin{proof}
Choose a lift, again denoted by $T$,  in $\U(n,1)$ such that $T=KP$, where $K$ is elliptic with characteristic polynomial $(x-\lambda)^3(x-1)^{n-2}$ and $P \in \SU(n,1)$ is a ellipto-parabolic with null eigenvalue $1$. Then $\C^{n,1}$ has a $P$-invariant decomposition
$\C^{n,1}={\mathbb U} \oplus \W$, where $\dim \u=3$, $P|_{{\mathbb U}}$ has minimal polynomial $(x-1)^3$ and, $\dim \W=n-2$, $P|_{\W} \in \SU(n-2)$. Now by \lemref{vt}, $P|_{{\mathbb U}}=i_1 i_2$, where $i_1$, $i_2$ are involutions and by \lemref{lem3}, $P|_{\W}$ is a product of two reversible elements $P|_{\W}=r_1 r_2$. Let $R_1=i_1 \oplus r_1$ and $R_2=i_2 \oplus r_2$. Then $P=R_1 R_2$. Note that, $R_1$ and $R_2$ are reversible elements in $\U(n,1)$ and hence by \thmref{rch}, each of them is a product of two involutions. The elliptic element $K$ is clearly a complex reflection that fixes a totally geodesic two dimensional subspace of $\h^n$.  This completes the proof.\end{proof}
\begin{cor}
Let $T$ be an ellipto-parabolic in $\SU(n,1)$ with null eigenvalue $1$. Then $T$ is a product of four involutions. 
\end{cor}
\subsubsection*{\bf Proof of \thmref{th2}} $\;$ 

\medskip Combining \corref{ce}, \corref{ch}, \lemref{vt}, \lemref{et} and \lemref{ep}, we have \thmref{th2}. 
\section{Product of anti-holomorphic involutions}\label{choi}
We have seen $\PU(n,1)$ as the group of isometries of $\h^n$. In the ball model, an element $A \in \PU(n,1)$ is an holomorphic isometry of $\h^n$. However, the \emph{real reflection} on $\C^{n,1}$, given by $c: v \mapsto \overline v$, also induces an isometry. The group $\PU(n,1)$, along with $c$ generate the full group $\I(\h^n)$ of isometries of $\h^n$. Thus $\PU(n,1)$ is an index two subgroup of $\I(\h^n)$. An anti-holomorphic isometry is given by $A\circ c$, where $A \in \U(n,1)$. For simplicity we write $A \circ c(v)=A \bar v$. We give a short proof of the following well-known result. 
 \begin{theorem}
Every holomorphic isometry of $\h^n$ is a product of two anti-holomorphic involutions.
\end{theorem}
\begin{proof} Let $T \in \U(n,1)$ be elliptic. Then $\C^{n,1}$ has a $T$-invariant decomposition into eigenspaces
$$\C^{n,1}=\V_{\lambda_1} \oplus \V_{\lambda_2} \oplus \cdots \oplus \V_{\lambda_k},$$
where $\lambda_i$ are the eigenvalues, $|\lambda_i|=1$ and $\V_{\lambda_1}$ is time-like. We define involutions $\alpha$ and $\beta$ on $\C^{n,1}$ by defining it on each of the eigenspaces. For $v \in \V_{\lambda}$, define $\alpha|_{\V_{\lambda_i}}(v)=\bar v$ and $\beta|_{\V_{\lambda_i}}(v)=\lambda_i \bar v$. Then $T=\beta \alpha$. 

Let $T \in \U(n,1)$ be hyperbolic. In this case also $T$ has a decomposition into eigenspaces and, 
by defining $\alpha$ and $\beta$ similarly as in the elliptic case, it is possible to write $T=\beta \alpha$.

Let $T \in \U(n,1)$ be a unipotent element. Then it has a minimal polynomial $(x-1)^2$ or $(x-1)^3$. Suppose $T$ has minimal polynomial $(x-1)^2$. Up to conjugacy, we can choose $T$, cf. \cite[Lemma 3.4.2]{chen}, such that for null vectors $u$ and $v$, $T|_{\u}$ has the following form with respect to basis $\{u, v\}$: 
$$
T|_{\u}=\begin{pmatrix} 1 & i \\ 0 & 1\end{pmatrix},$$
where $\u$ is the non-degenerate $T$-invariant subspace of $\C^{n,1}$ generated by $u, v$. 
The restriction of $\langle\cdot, \cdot\rangle$ to $\u$ has 
signature $(1,1)$ and $T|_{\u^{\perp}}$ is the identity map. 

For $w \in \u$,  define $\mu(w)= T \bar w$ and $\nu(w)=\bar w$. Hence $\mu^2(w)=T|_{\u} \overline T|_{\u}~w=w$. Thus, $\mu$ and $\nu$ are involutions and  
$T|_{\u}=\nu \mu$. Extending $\mu$ and $\nu$ to the whole of $\C^{n,1}$ by composing the map $c$ on $\u^{\perp}$, we have the required involution. 

If $T$ has minimal polynomial $(x-1)^3$, then it follows from \cite[Theorem 4.1]{gp} that $T$ is a product of two involutions. Further, up to conjugacy, the involutions may be chosen to be elements in ${\rm O}(n,1)$. Hence those involutions can be extended easily to anti-holomorphic ones by adjoining the real reflection $c$. 

When $T$ is non-unipotent, then we can choose a lift in $\U(n,1)$ such that the null eigenvalue is $1$ and consequently, $\C^{n,1}$ has a decomposition $\C^{n,1}=\u \oplus \W$ as in \eqnref{pe}.  Accordingly, we can construct anti-holomorphic involutions on each of these subspaces as above. The desired involutions are obtained by taking orthogonal sum of them. 

This proves the result. \end{proof}

 \begin{cor}
  Every holomorphic isometry of $\h^n$ is a commutator in the isometry group of $\h^n$. 
 \end{cor}
 \begin{proof}
  Let $A\in \U(n,1)$. It is easy to see that there is $B$ in $\U(n,1)$ such that $B^2=A$. By the above theorem $B=\alpha\beta$ where $\alpha$ and $\beta$ are anti-holomorphic involutions. Then $A=[\alpha,\beta]$.
 \end{proof}

\subsubsection*{Acknowledgements}  The authors thank Julien Paupert and  Pierre Will for their comments and suggestions on a first draft of this paper and for sharing their investigations. Thanks also to Anupam Singh for comments. 

 This work is a part of the second named author's MS dissertation at IISER Mohali. She would like to express gratitude to IISER Mohali  for providing support throughout her BS-MS course there.  She thanks John Parker for many discussions during an NCM-workshop at IISER Mohali around December 2014.

\end{document}